\documentclass[graybox,envcountsect]{svmult}

\usepackage{srcltx}%
\usepackage{mathptmx}       
\usepackage{helvet}         
\usepackage{courier}        
\usepackage{type1cm}        
\usepackage{makeidx}         
\usepackage{graphicx}        
\usepackage{multicol}        
\usepackage[bottom]{footmisc}

\usepackage{amsfonts,amsmath}

\makeindex             

\def \Z {{\mathbb Z}}
\def \R {{\mathbb R}}

\def \N {{\mathbb N}}

\def \da {\downarrow}

\def \PP {{\mathbb P}}
\def \EE {{\mathbb E}}

\def \ES {{\rm E}}

\def \Sp {{\rm Sp}}

\def \cL {{\mathcal L}}

\def \cM {{\mathcal M}}

\def \cI {{\mathcal I}}
\def \cH {{\mathcal H}}
\def \cB {{\mathcal B}}

\newcommand {\refeq}[1] {(\ref{#1})}
\newcommand {\cro}[1] {\left[ {#1} \right]}
\newcommand {\pare}[1] {\left( {#1} \right)}
\newcommand {\ind}[1] {\hbox{ 1\hskip -3pt I}_{#1}}
\newcommand {\bra}[1] {\left\langle {#1} \right\rangle}
\newcommand {\nor}[1] { \left\| {#1} \right\|}
\newcommand {\acc}[1] {\left\{ {#1} \right\}}
\newcommand {\va}[1] {\left| {#1} \right|}
\def\one{\rlap{\mbox{\small\rm 1}}\kern.15em 1}

\begin{document}

\title*{Parabolic Anderson model with a finite number of moving catalysts}
\author{F.\ Castell, O.\ G\"un and G.\ Maillard}
\institute{F.\ Castell \at CMI-LATP, Universit\'e de Provence,
39 rue F. Joliot-Curie, F-13453 Marseille Cedex 13, France, \email{castell@cmi.univ-mrs.fr}
\and O.\ Gun \at CMI-LATP, Universit\'e de Provence,
39 rue F. Joliot-Curie, F-13453 Marseille Cedex 13, France, \email{gun@cmi.univ-mrs.fr}
\and G.\ Maillard \at CMI-LATP, Universit\'e de Provence,
39 rue F. Joliot-Curie, F-13453 Marseille Cedex 13, France, \email{maillard@cmi.univ-mrs.fr},
and EURANDOM, P.O.\ Box 513, 5600 MB Eindhoven, The Netherlands}
%
%
\maketitle

\abstract*{}

\abstract{We consider the parabolic Anderson model (PAM) which is given by the equation 
$\partial u/\partial t = \kappa\Delta u + \xi u$ with $u\colon\, \Z^d\times [0,\infty)\to \R$, 
where $\kappa \in [0,\infty)$ is the diffusion constant, $\Delta$ is the discrete Laplacian, 
and $\xi\colon\,\Z^d\times [0,\infty)\to\R$ is a space-time random environment. The solution 
of this equation describes the evolution of a ``reactant'' $u$ under the influence 
of a ``catalyst'' $\xi$.
\newline\indent
In the present paper we focus on the case where $\xi$ is a system of $n$ independent
simple random walks each with step rate $2d\rho$ and starting from the origin. We 
study the \emph{annealed} Lyapunov exponents, i.e., the exponential growth rates of 
the successive moments of $u$ w.r.t.\ $\xi$ and show that these exponents, as a function
of the diffusion constant $\kappa$ and the rate constant $\rho$, behave
differently depending on the dimension $d$.
In particular, we give a description of the intermittent behavior of the system
in terms of the annealed Lyapunov exponents, depicting how the total mass of $u$
concentrates as $t\to\infty$.
Our results are both a generalization and an extension of the work of G\"artner and
Heydenreich \cite{garhey06}, where only the case $n=1$ was investigated.}


\section{Introduction}
\label{S1}


\subsection{Model}
\label{S1.1}

The parabolic Anderson model (PAM) is the partial differential equation
\begin{equation}
\label{pA}
\left\{
\begin{array}{rll}
\displaystyle\frac{\partial}{\partial t}u(x,t) &=& \kappa\Delta u(x,t) + \xi(x,t)u(x,t),\bigg.\\
u(x,0) &=& 1,
\end{array}
\qquad x\in\Z^d,\,t\geq 0\, .
\right.
\end{equation}
Here, the $u$-field is $\R$-valued, $\kappa\in [0,\infty)$ is the diffusion
constant, $\Delta$ is the discrete Laplacian acting on $u$ as
\begin{equation*}
\label{}
\Delta u(x,t) = \sum_{{y\in\Z^d} \atop {y\sim x}} [u(y,t)-u(x,t)]
\end{equation*}
($y\sim x$ meaning that $y$ is nearest neighbor of $x$), and
\begin{equation*}
\label{}
\xi = (\xi_t)_{t \geq 0} \quad\text{with}\quad \xi_t = \{\xi(x,t) \colon\,x\in\Z^d\}
\end{equation*}
is an $\R$-valued random field that evolves with 
time and that drives the equation.

One interpretation of (\ref{pA}) comes from population dynamics by considering
a system of two types of particles $A$ and $B$. $A$-particles represent ``catalysts'', 
$B$-particles represent ``reactants'' and the dynamics is  subject to the following rules:
\begin{itemize}
\item
$A$-particles evolve independently of $B$-particles according to a prescribed 
dynamics with $\xi(x,t)$ denoting the number of $A$-particles at site $x$ at time $t$;
\item
$B$-particles perform independent simple random walks at rate $2d\kappa$ and split
into two at a rate that is equal to the number of $A$-particles present at
the same location;
\item
the initial configuration of $B$-particles is that there is exactly one particle at each 
lattice site.
\end{itemize}
Then, under the above rules, $u(x,t)$ represents the average number of 
$B$-particles at site $x$ at time $t$ conditioned on the evolution of the 
$A$-particles. 

It is possible to add that $B$-particles die at rate $\delta\in [0,\infty)$.
This leads to the trivial transformation $u(x,t) \to u(x,t)e^{\delta t}$. 
We will hereafter assume that $\delta=0$.
It is also possible to add a coupling constant $\gamma\in (0,\infty)$
in front of the $\xi$-term in (\ref{pA}), but this can be reduced to 
$\gamma=1$ by a scaling argument.

In what follows, we focus on the case where
\begin{equation}
\label{xidef}
\xi(x,t)=\sum_{k=1}^{n}\delta_{x}\big(Y_{k}^{\rho}(t)\big)
\end{equation}
with $\{Y_k^\rho:\;1\leq k \leq n\}$ a family of $n$ independent simple random walks, 
where for each  $k\in\{1,\ldots, n\}$, $Y_{k}^{\rho}=(Y_{k}^{\rho}(t))_{t\geq 0}$
is a simple random walk with step rate $2d\rho$ starting from the origin. 
We write $\PP_{0}^{\otimes n}$ and $\EE_{0}^{\otimes n}$ to denote respectively the 
law and the expectation of the family of $n$ independent simple random walks 
$\{Y_{k}^{\rho}\colon 1\leq k\leq n\}$ where initially all of the walkers are located at $0$.

The rest of the section is organized as follows. In Section \ref{S1.2}, we define the
\emph{annealed Lyapunov exponents\/} and introduce the \emph{intermittency\/} 
phenomenon. In Section \ref{S1.3}, we review some related models from the literature. 
In Section \ref{S1.4}, we state our main results, and finally, in Section \ref{S1.5}, we 
give some further comments and add few results and conjectures.


\subsection{Lyapunov exponents and intermittency}
\label{S1.2}

Our focus will be on the \emph{annealed} Lyapunov exponents that describe the 
exponential growth rate of the successive moments of the solution of (\ref{pA}). 

By the Feynman-Kac formula, the solution of (\ref{pA}) reads
\begin{equation}
\label{fey-kac1}
u(x,t) = \ES_{\,x}\left(\exp\left[\int_0^t \xi\left(X^\kappa(s),t-s\right)\,ds\right]\right),
\end{equation}
where $X^\kappa=(X^\kappa(t))_{t \geq 0}$ is the simple random walk on $\Z^d$ with 
step rate $2d\kappa$ and $\ES_{\,x}$ denotes expectation with respect to $X^\kappa$ 
given $X^\kappa(0)=x$. The connection between the parabolic Anderson equation 
(\ref{pA}) with random time-independent potential $\xi$ and the Feyman-Kac functional 
(\ref{fey-kac1}) is well understood (see e.g.\ G\"artner and Molchanov \cite{garmol90}) 
and can be easily extended to the time-dependent potential setting. Taking into account 
our choice of catalytic medium in (\ref{xidef}) we define $\Lambda_p(t)$ as
\begin{eqnarray}
\label{lyapdef2}
\Lambda_p(t)
&=&  \frac{1}{t} \log\EE_{\,0}^{\otimes n}\big([u(x,t)]^p\big)^{1/p}\nonumber\\
&=&  \frac{1}{pt} \log \big(\EE_{\,0}^{\otimes n}\otimes\ES_{\,x}^{\otimes p}\big)
\Bigg(\exp\Bigg[\sum_{j=1}^{p}\sum_{k=1}^{n}\int_0^t
\delta_{0}\big(X_{j}^{\kappa}(s)-Y_{k}^{\rho}(t-s)\big)\,ds
\Bigg]\Bigg),
\end{eqnarray}
where $\{X_{j}^\kappa\colon 1\leq j\leq p\}$ is a family of $p$ independent copies of 
$X^\kappa$ and $\ES_{\,x}^{\otimes p}$ stands for the expectation of this family 
with $X_{j}^{\kappa}(0)=x$
for all $j$.

If the last quantity admits a limit as $t\to\infty$ we define 
\begin{equation}
\label{aLyapdef}
\lambda_p := \lim_{t\to\infty} \Lambda_p(t)
\end{equation}
to be the $p$-th (annealed) Lyapunov exponent of the solution $u$ of the PAM (\ref{pA}).

We will see in Theorem
\ref{th1} that the limit in (\ref{aLyapdef}) exists and is independent of $x$. Hence, we 
suppress $x$ in the notation. However,  $\lambda_p$ is clearly a function of $n$, $d$, 
$\kappa$ and $\rho$. In what follows, our main focus will be to analyze the dependence 
of $\lambda_{p}$ on the parameters $n$, $p$, $\kappa$ and $\rho$, therefore we will 
often write $\lambda_p^{(n)}(\kappa,\rho)$.

In particular, our main subject of interest will be to draw the qualitative picture of 
{\it intermittency} for these systems. First, note that by the moment inequality we have 
\begin{equation}
\lambda_p^{(n)}\geq\lambda_{p-1}^{(n)},
\end{equation}
for all $p\in\N\setminus\{1\}$. 
The system (or the solution of the system) (\ref{pA}) is said to be
{\em $p$-intermittent\/} if the above inequality is strict, namely,
\begin{equation}
\label{p-int}
\lambda_p^{(n)}>\lambda_{p-1}^{(n)}.
\end{equation}
The system is \emph{fully intermittent} if (\ref{p-int}) holds for all 
$p\in\mathbb{N}\setminus\{1\}$. We will sometimes say that the system is 
\emph{partially intermittent} if it is $p$-intermittent for some $p\in\N\setminus\{1\}$. 

Also note that, using H\"older's inequality, $p$-intermittency implies
$q$-intermittency for all $q\geq p$ (see e.g. \cite{garhey06}, Lemma 3.1).
Thus, for any fixed $n\in\N$, $p$-intermittency in fact implies that
\begin{equation*}
\lambda_{q}^{(n)} > \lambda_{q-1}^{(n)}
\quad \forall q\geq p \, ,
\end{equation*}
and $2$-intermittency means full intermittency.

Geometrically, intermittency corresponds to the solution being \emph{asymptotically 
concentrated} on a thin set, which is expected to consist of ``islands'' located far from 
each other (see \cite{garkon05}, Section 1 and references therein for more details).
Here, due to the lack of ergodicity, such a geometric picture of intermittency is not
available. Nevertheless,  (\ref{p-int}) can still be interpreted as the $p$-th moment of
$u$ being generated by some exponentially rare event (see \cite{garhey06}, Section 
1.2 for a more detailed analysis).


\subsection{Literature}
\label{S1.3}

The behavior of the \emph{annealed} Lyapunov exponents and particularly the problem
of intermittency for the PAM in a \emph{space-time} random environment was subject to
various studies. Carmona and Molchanov \cite{carmol94} obtained an essentially complete
qualitative description of the annealed Lyapunov exponents and intermittency  when $\xi$ 
is white noise, i.e.,
\begin{equation*}
\label{}
\xi(x,t) = \frac{\partial}{\partial t}\,W(x,t) \, ,
\end{equation*}
where $W=(W_t)_{t \geq 0}$ with $W_t=\{W(x,t)\colon\,x\in\Z^d\}$ is a field of independent 
Brownian motions.  In particular, it was shown that $\lambda_{1}=1/2$ for all $d\geq 1$ and, 
$\lambda_p>1/2$ for $p\in\N\setminus\{1\}$ in $d=1,2$. It is also proved that for $d \geq 3$ 
there exist $0<\kappa_2\leq\kappa_3\leq\ldots$ 
satisfying
\begin{equation*}
\label{Lyaref}
\lambda_p(\kappa)-\frac12
\left\{
\begin{array}{lr}
>0,   &\text{for } \kappa \in \big[0,\kappa_p\big),\Big.\\
= 0,  &\text{for } \kappa \in \big[\kappa_p,\infty\big),
\end{array}
\quad p \in \N\setminus\{1\} \,.
\right.
\end{equation*}
Further refinements on the behavior of the Lyapunov exponents were 
obtained in Greven and den Hollander \cite{grehol07}.
Upper and lower bounds on $\kappa_p$ were derived, and the asymptotics of
$\kappa_p$ as $p \to \infty$ was computed. In addition, it was proved that the $\kappa_p$'s 
are distinct for $d$ large 
enough.

More recently various models where $\xi$ is \emph{non-Gaussian} were investigated.
Kesten and Sidoravicius \cite{kessid03}  and G\"artner and den Hollander \cite{garhol06}
 considered the case where $\xi$ is given by a Poisson field of independent simple 
random walks. In \cite{kessid03}, the survival versus extinction of the system is studied. 
In \cite{garhol06}, the moment asymptotics were studied and a partial picture of intermittency, 
depending on the parameters $d$ and $\kappa$, was obtained. The case where $\xi$ is a 
single random walk --corresponding to $n=1$ case in our setting-- was studied by G\"artner 
and Heydenreich \cite{garhey06}. Analogous results to those contained in Theorems \ref{th1}, 
\ref{th2} and Corollary \ref{th5}(i) were obtained. 

The investigation of annealed Lyapunov behavior and intermittency was extented to
non-Gaussian and \emph{space correlated} potentials in G\"artner, den Hollander and
Maillard, in \cite{garholmai07} and \cite{garholmai09}, for the case where $\xi$ is an 
exclusion process with symmetric random walk transition kernel, starting form a Bernoulli 
product measure. Later G\"artner, den Hollander and Maillard \cite{garholmai10}, and 
Maillard, Mountford and Sch\"opfer \cite{maimousch11}, studied the case where $\xi$ is a 
voter model starting either from Bernoulli product measure or from equilibrium  
(see G\"artner, den Hollander and Maillard \cite{garholmai08HvW}, for an overview).


\subsection{Main results}
\label{S1.5}

Our first theorem states that the Lyapunov exponents exist and behave nicely as 
a function of $\kappa$ and $\rho$. It will be proved in Section \ref{S2}.

\begin{theorem}[Existence and first properties]
\label{th1}
Let $d \geq 1$ and $n,p\in\N$.\\
(i) For all $\kappa,\rho\in[0,\infty)$, the limit in {\rm (\ref{aLyapdef})} exists, is finite,
and is independent of $x$ if $(\kappa,\rho) \ne (0,0)$. \\
(ii) On $[0,\infty)^{2}$, $(\kappa,\rho)\mapsto\lambda_p^{(n)}(\kappa,\rho)$
is continuous, convex and non-increasing in both $\kappa$ and $\rho$.
\end{theorem}
Let $G_d(x)$ be the Green function at lattice site $x$ of simple random walk 
stepping at rate $2d$ and
\begin{equation}
\label{mudef}
\mu(\kappa)=\sup\Sp(\kappa\Delta+\delta_{0})
\end{equation}
be the supremum of the spectrum of the operator $\kappa\Delta+\delta_{0}$ in
$l^2(\Z^d)$. It is well-known that (see e.g.\ \cite{garhol06}, Lemma 1.3) 
$\Sp(\kappa\Delta+\delta_{0})=[-4d\kappa,0]\cup\{\mu(\kappa)\}$
with
\begin{equation}
\label{muprop}
\mu(\kappa)
\left\{
\begin{array}{lr}
=0, &\quad\text{if } \kappa\geq G_d(0),\\
>0, &\quad\text{if } \kappa<G_d(0).
\end{array}
\right.
\end{equation}
Furthermore, $\kappa\mapsto \mu(\kappa)$ is continuous, non-increasing and 
convex on $[0,\infty)$, and strictly decreasing on $[0,G_{d}(0)]$. 

The next theorem gives the limiting behavior of $\lambda_{p}^{(n)}$ as $\kappa\da 0$ 
and $\kappa\to\infty$, and describes a region of $\kappa$ where $\lambda_{p}^{(n)}=0$. 
Note that by symmetry, 
$\lambda_{p}^{(n)}(\kappa,\rho)=\frac{n}{p}\lambda_{n}^{(p)}(\rho,\kappa)$,
for all $n,p\in\N$ and $\kappa,\rho\in[0,\infty)$. Therefore, the $\kappa$-dependence 
described below can be transcribed in terms of $\rho$-dependence. 

\begin{theorem}[$\kappa$- $\rho$-dependence]
\label{th2}
Let $n,p\in\N$ and $\rho\in[0,\infty)$.\\
(i) For all $d\geq 1$,
$\lim_{\kappa\da 0}\lambda_{p}^{(n)}(\kappa,\rho)=\lambda_{p}^{(n)}(0,\rho)
=n\mu(\rho/p)$.\\
(ii) If $1\leq d\leq 2$, then $\lambda_p^{(n)}(\kappa,\rho)>0$  for all $\kappa\in[0,\infty)$.
Moreover, $\kappa\mapsto \lambda_p^{(n)}(\kappa,\rho)$ is strictly decreasing with
$\lim_{\kappa\to\infty}\lambda_p^{(n)}(\kappa,\rho)=0$ 
(see Fig.~{\rm\ref{fig-lambda01}}).\\
(iii) If  $d\geq 3$, then $\lambda_p^{(n)}(\kappa,\rho)=0$ 
for all $\kappa \in [nG_d(0),\infty)$ (see Fig.~{\rm\ref{fig-lambda02}}).
\end{theorem}

Our next result describes the limiting behavior of $\lambda_p^{(n)}$ as $p\to\infty$ and $n\to\infty$. 

\begin{theorem}[$n$- $p$-dependence]
\label{th3}
Let $d\geq 1$ and $\kappa, \rho\in[0,\infty)$.\\
(i) For all $n\in\N$, $\lim_{p\to\infty}\lambda_{p}^{(n)}(\kappa,\rho)=n\mu(\kappa/n)$
(see Fig.~{\rm\ref{fig-lambda01}--\ref{fig-lambda02}});\\
(ii) For all $p > \rho/G_d(0)$, $\lim_{n\to\infty}\lambda_{p}^{(n)}(\kappa,\rho)=+\infty$;\\
(iii) For all $p \le \rho/G_d(0)$ and $n\in\N$, $\lambda_{p}^{(n)}(\kappa,\rho)=0$.
\end{theorem}

By part (ii) of Theorem \ref{th1}, $\lambda_p^{(n)}(\kappa,\rho)$ is non-increasing 
in $\kappa$. Hence, we can define $\big\{\kappa_{p}^{(n)}(\rho) \colon p\in\N\big\}$ 
as the non-decreasing sequence of critical $\kappa$'s for which
\begin{equation}
\label{kappa-crit}
\lambda_p^{(n)}(\kappa,\rho)
\left\{
\begin{array}{lr}
>0,   &\text{for } \kappa \in \big[0,\kappa_p^{(n)}(\rho)\big),\\
= 0,  &\text{for } \kappa \in \big[\kappa_p^{(n)}(\rho),\infty\big),
\end{array}
\quad p \in \N \, .
\right.
\end{equation}
As a consequence of Theorems \ref{th1} and \ref{th2} we have,
\begin{equation}
\begin{cases}
\;\kappa_{p}^{(n)}(\rho)=\infty, &\text{if } 1\leq d\leq2,\\
\;0<\kappa_{p}^{(n)}(\rho)<\infty, &\text{if } d\geq 3 \text{ and } 
p > \rho/G_d(0), \\
\;\kappa_{p}^{(n)}(\rho) = 0, &\text{if } d\geq 3 \text{ and } 
p \le \rho/G_d(0).
\end{cases}
\end{equation}
Our fourth theorem, which gives bounds on $\kappa_{p}^{(n)}(\rho)$ for $d \ge 3$, 
will be proved in
Section \ref{S4}. For this theorem we need to define the inverse of the function 
$\mu(\kappa)$. Note that by (\ref{mudef}) and (\ref{muprop}) we have $\mu(0)=1$ 
and $\mu(G_d(0))=0$. It is easy to see that $\mu(\kappa)$ restricted to the domain 
$[0,G_d(0)]$ is invertible with an inverse function $\mu^{-1}\colon [0,1]\to [0,G_d(0)]$. 
We extend $\mu^{-1}$ to $[0,\infty)$ by declaring $\mu^{-1}(t)=0$ for $t>1$. 
Denote
\begin{equation}
\alpha_d = \frac{G_d(0)}{2d \|G_d\|_2^2} \in [0,\infty) \, ,
\end{equation}
where  $\|G_d\|_2$ is the $l_2$ norm of $G_d$. Since $\|G_d\|_2 < \infty$ if and only if $d\ge 5$,
$\alpha_d = 0$ for $d \in \acc{3,4}$. 
\begin{theorem}[Critical $\kappa$'s]
\label{th4}
Let $n, p\in\N$.\\
(i) If $d\geq 3$, then
 $\rho \in[0,\infty) \mapsto \kappa_{p}^{(n)}(\rho)$
is a continuous, non-increasing and convex function such that 
\begin{equation}
\label{kappabds}
\max\bigg(\frac{n}{4d}\, \mu(\rho/p), n \mu^{-1}(4d\rho/p) \bigg)
\leq \kappa_{p}^{(n)}(\rho) \leq 
nG_d(0)  \bigg(1-\frac{\rho}{pG_d(0)}\bigg)_+. 
\end{equation}
\\
(ii) If $d\ge 5$, then
\begin{equation}
\label{kappabds2}
 \kappa_{p}^{(n)}(\rho)  \geq \bigg(n G_d(0) - \rho \frac{n}{p \alpha_d}\bigg)_{+} . 
\end{equation}
(iii) If $d\ge 5$ and $p\in\N\setminus\{1\}$ are such that $\alpha_d > \frac{p-1}{p}$, 
then
\begin{equation}
\label{kappa_croit}
\kappa_{p-1}^{(n)}(\rho) < \kappa_{p}^{(n)}(\rho)
\quad \forall \rho \in (0, p G_d(0)).
\end{equation} 
\end{theorem}

Note that the condition $\alpha_d > \tfrac{p-1}{p}$ is always true if $d$ is large enough
by the following lemma, whose proof is given in the appendix.

\begin{lemma}
\label{lm}
\label{alpha}
If $d \ge 3$, then $\alpha_d\le 1$ and $ \lim_{d \rightarrow \infty} \alpha_d = 1$.
\end{lemma}

As a consequence of the previous statements, our next result gives some general 
intermittency properties for all dimensions, and describes several regimes in the 
intermittent behavior of the system. 

\begin{corollary}[Intermittency]
\label{th5}
Let $n\in\N$.\\
(i) If $d\geq 1$, then (see Fig.~{\rm\ref{fig-lambda02}})
\begin{itemize}
\item[-]
for $\kappa\in [0, nG_d(0))$ there exists $p\geq 2$ such that the system is 
 $p$-intermittent;
 \item[-]
for $\kappa\in [nG_d(0),\infty)$ the system is not intermittent.
\end{itemize}
(ii) Fix $p\in\N\setminus\{1\}$. If $d$ is large enough (such that $\alpha_d > (p-1)/p$ ), 
then for all $q\in\{2,\ldots,p\}$, $\rho\in[0, q G_d(0))$ and 
$\kappa\in(\kappa_{q-1}^{(n)}(\rho),\kappa_{q}^{(n)}(\rho))$, the system is $q$-intermittent
(see Fig.~{\rm\ref{fig-lambda03}}).
\end{corollary}
Note that since $G_d(0)=\infty$ for $d=1,2$ Corollary \ref{th5}(i) implies that for 
dimensions $1$ and $2$ the system is always $p$-intermittent for some $p$. Some other 
partial results about intermittency are given in section \ref{S1.4} (see also figures).

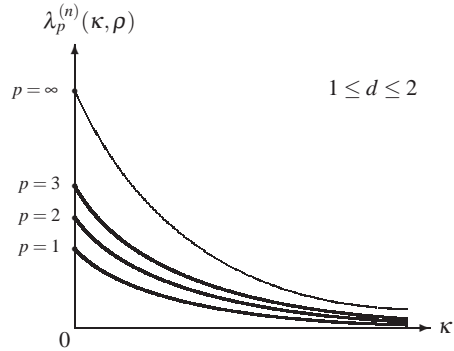
\begin{figure}[ht]
\vspace{3cm}
\begin{center}
\setlength{\unitlength}{0.21cm}
\begin{picture}(0,12)(16,0)
\put(0,2){\vector(1,0){22.5}}
\put(0,2){\vector(0,1){18}}
{\thicklines
\qbezier(0,7)(4,2.6)(21,2.2)
}
{\thicklines
\qbezier(0,9)(4,3.2)(21,2.4)
}
{\thicklines
\qbezier(0,11)(4,3.6)(21,2.6)
}
\qbezier[200](0,17)(6,3.6)(21,3.2)
\put(-1,0.8){$0$}
\put(23,1.7){$\kappa$}
\put(16,16.8){$1\leq d\leq 2$}
\put(-2,21){$\lambda_p^{(n)}(\kappa,\rho)$}
\put(-3.6,6.8){{\scriptsize$p=1$}}
\put(-3.6,8.8){{\scriptsize$p=2$}}
\put(-3.6,10.8){{\scriptsize$p=3$}}
\put(-4,16.8){{\scriptsize$p=\infty$}}
\put(0,7){\circle*{.35}}
\put(0,9){\circle*{.35}}
\put(0,11){\circle*{.35}}
\put(0,17){\circle*{.35}}
\end{picture}
\caption{\small For $1\leq d\leq 2$, the system is partially intermittent. 
Full intermittency is conjectured, and proved for $n=1,2$.}
\label{fig-lambda01}
\end{center}
\end{figure}



\begin{figure}[ht]
\vspace{3cm}
\begin{center}
\setlength{\unitlength}{0.21cm}
\begin{picture}(0,12)(16,0)
\put(0,3.5){\vector(1,0){29}}
\put(0,3.5){\vector(0,1){18}}
{\thicklines
\qbezier(0,8.5)(2,4.1)(7,3.5)
}
{\thicklines
\qbezier(7,3.5)(12,3.5)(29,3.5)
}
{\thicklines
\qbezier(0,10.5)(3,4.1)(10,3.5)
}
{\thicklines
\qbezier(0,12.5)(4,4.1)(13,3.5)
}
\qbezier[200](0,18.5)(6,4.1)(21,3.5)
%
\put(-1,2.3){$0$}
\put(30,3.2){$\kappa$}
\put(-2,22.5){$\lambda_p^{(n)}(\kappa, \rho)$}
\put(18,18.3){$d\geq 3$, $\rho<G_d(0)$}
\put(-3.6,8.3){{\scriptsize$p=1$}}
\put(-3.6,10.3){{\scriptsize$p=2$}}
\put(-3.6,12.3){{\scriptsize$p=3$}}
\put(-4,18.3){{\scriptsize$p=\infty$}}
\put(0,8.5){\circle*{.35}}
\put(0,10.5){\circle*{.35}}
\put(0,12.5){\circle*{.35}}
\put(0,18.5){\circle*{.35}}
\put(7,3.5){\circle*{.35}}
\put(10,3.5){\circle*{.35}}
\put(13,3.5){\circle*{.35}}
\put(21,3.5){\circle*{.35}}
\put(6.5,2){{\scriptsize$\kappa_{1}^{(n)}$}}
\put(9.5,2){{\scriptsize$\kappa_{2}^{(n)}$}}
\put(12.5,2){{\scriptsize$\kappa_{3}^{(n)}$}}
\put(20.2,2){{\scriptsize$nG_d(0)$}}
\put(0,0){$|$}
\put(1.8,0){$\longleftarrow$}
\put(4.6,0){{\scriptsize $A$}}
\put(6,0){$\longrightarrow$}
\put(10,0){$|$}
\put(12,0){$\longleftarrow$}
\put(15.2,0){{\scriptsize $B$}}
\put(17,0){$\longrightarrow$}
\put(21,0){$|$}
\put(22,0){$\longleftarrow$}
\put(24.8,0){{\scriptsize $C$}}
\put(26.5,0){$\longrightarrow$}
\end{picture}
\caption{\small For $d\geq 3$ and $\rho<G_d(0)$, the system is partially intermittent 
on $A\cup B$ and not intermittent on $C$. Full intermittency on $A$ is conjectured, and 
proved for $n=1,2$.}
\label{fig-lambda02}
\end{center}
\end{figure}
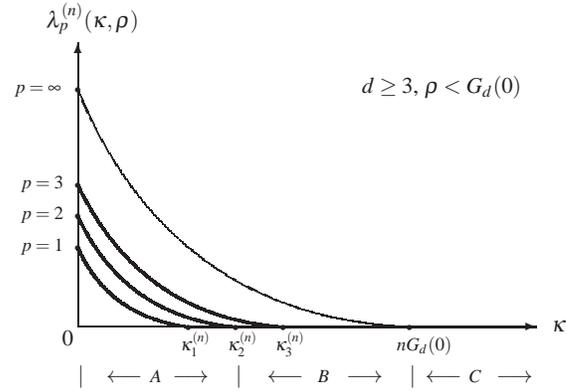



\begin{figure}[ht]
\vspace{3cm}
\begin{center}
\setlength{\unitlength}{0.22cm}
\begin{picture}(0,12)(16,0)
\put(0,3.5){\vector(1,0){29}}
\put(0,3.5){\vector(0,1){18}}
{\thicklines
\qbezier(0,18.5)(2,6)(6.5,3.5)
}
{\thicklines
\qbezier(0,18.5)(3,8.5)(10,3.5)
}
{\thicklines
\qbezier(0,18.5)(6,8)(20,3.5)
}
{\thicklines
\qbezier(0,18.5)(8,10)(25,3.5)
}
\qbezier[200](0,18.5)(14,18.5)(28,18.5)
%
\put(-1,2.3){$0$}
\put(30,3.2){$\rho$}
\put(-1,22.5){$\kappa$}
\put(-4,18.3){{\scriptsize$nG_{d}(0)$}}
\put(0,18.5){\circle*{.35}}
\put(6.5,3.5){\circle*{.35}}
\put(10,3.5){\circle*{.35}}
\put(20,3.5){\circle*{.35}}
\put(25,3.5){\circle*{.35}}
\put(6.5,2){{\scriptsize$G_d(0)$}}
\put(9.5,2){{\scriptsize$2G_d(0)$}}
\put(14.5,2){{\scriptsize$\cdots$}}
\put(17.4,2){{\scriptsize$(p-1)G_d(0)$}}
\put(23.8,2){{\scriptsize$pG_d(0)$}}
\put(1,7){{?}}
\put(3.5,7){{\scriptsize$2$-int.}}
\put(7.5,7){{$\cdots$}}
\put(12.6,7){{\scriptsize$p$-int.}}
\put(20,7){{$\cdots$}}
\put(12.6,20){{\scriptsize no intermittency}}
\end{picture}
\vspace{-.5cm}
\caption{\small Phase diagram of intermittency when $d$ is large enough.
The bold curves represent $\rho\in[0,\infty)\mapsto \kappa_{q}^{(n)}(\rho)$,
$q=1,\cdots,p$. In the ``?'' region, full intermittency is proved in a small neighborhood 
of $0$.}
\label{fig-lambda03}
\end{center}
\end{figure}
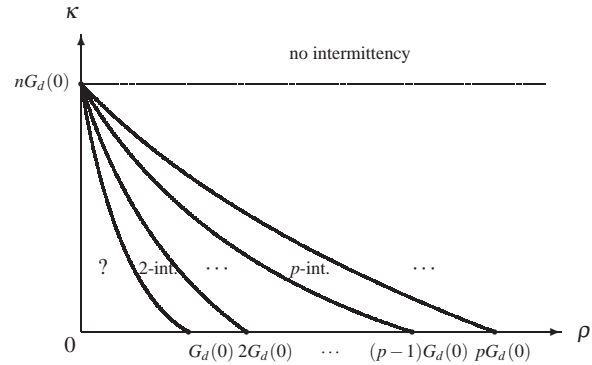



\subsection{Discussion}
\label{S1.4}

Our results can be extended to various different random medium. For example, consider 
the system of catalysts given by a collection of independent random walks where there is 
one walker starting from each site of a large box. More precisely, let $D_R$ denote the box 
in $\Z^d$ with side length $R$. Consider the random medium
\begin{equation*}
\xi(x,t)=\sum_{k\in D_R}\delta_{x}(Y_k^{\rho}(t))
\end{equation*}
with $\{Y_k^\rho:\;k\in D_R\}$ a family of $R^d$ simple random walks, where for each 
$k\in D_R$, $Y_k^{\rho}$ is a simple random walk with step rate $2d\rho$ starting from 
$Y_k^{\rho}(0)=k$. For a fixed size box, there is a positive probability that all the random 
walks meet at the origin in finite time. Then, it is easy to see that the Lyapunov exponents 
are the same as in the case of $n$ independent random walks starting from the origin 
where $n=R^d$. An interesting set up would be case where the length of the initial box 
grows with time. A natural question arises as whether the large time limit would be related 
to the case of Poisson field of simple random walks, considered in \cite{garhol06}, or it 
would have different behavior depending on how fast the size of the box grows with time.   

Let us  now  discuss some facts about the intermittent picture. First of all, as one can 
easily guess from \refeq{lyapdef2}, $\lambda^{(n)}_p(\kappa,\rho)$ is the top of  the 
spectrum of the operator $\cL_{p}$ where for $f( x_1, \cdots, x_p, y_1,\cdots, y_n)$ 
in $l_2(\Z^{d(p+n)})$ $\cL_{p}$, is defined by:
\begin{equation}
\label{Lpdef}
\cL_p(f)= \kappa \sum_{k=1}^p \Delta_{x_k} f 
+ \rho \sum_{j=1}^n \Delta_{y_j} f + I_p f \, .
\end{equation}
Here
\begin{equation*}
I_pf(x_1, \cdots, x_p,y_1,\cdots, y_n) = \sum_{k=1}^n \sum_{j=1}^p 
\delta_0(x_j-y_k)f(x_1, \cdots, x_p,y_1,\cdots, y_n).
\end{equation*}
This is the meaning of equation \refeq{formvar} of Section \ref{S2} from which most of 
our results are derived. The following proposition links  full intermittency and  existence 
of an eigenfunction corresponding to $\lambda^{(n)}_{1}(\kappa, \rho)$.
\begin{proposition}
\label{eigenfunc}
If there exists $f \in l_2(\Z^{d(1+n)})$ with $\nor{f}_2=1$, such that 
$\cL_1(f)=\lambda^{(n)}_1(\kappa,\rho) f$, then 
$\lambda^{(n)}_2(\kappa,\rho) > \lambda^{(n)}_1(\kappa,\rho)$, and 
the system is fully intermittent.
\end{proposition}
Proposition \ref{eigenfunc} is proved in the appendix. The existence of an eigenfunction 
corresponding to $\lambda^{(n)}_{1}(\kappa, \rho)$ (and therefore full intermittency)
was proved in the following cases:
\begin{itemize}
\item $n=1,2$ and $\kappa + \rho < n G_d(0)$. This is done  in \cite{garhey06} for $n=1$, 
and in \cite{schwol10} for $n=2$.
\item  $n \ge 3$ and $4d(\rho n + \kappa) < 1$ in  \cite{schwol10}.
\end{itemize}
To prove these results, in  \cite{garhey06} and \cite{schwol10} 
$\lambda^{(n)}_{1}(\kappa, \rho)$ was expressed as the top of the spectrum of the 
operator $\cH = \cB + \sum_{j=1}^n \delta_0(z_j)$, where   $\cB$ is the generator 
of the Markov process 
$Z(t)= (X^{\kappa}_1(t) - Y^{\rho}_1(t), \cdots,X^{\kappa}_1(t) - Y^{\rho}_n(t))$
(see \refeq{lyapdef2}).  For $n=1$,  $\cH$ is just $(\kappa + \rho) \Delta + \delta_0$, 
which is a compact perturbation of $(\kappa + \rho) \Delta$. This fact easily implies 
the existence of an eigenfunction corresponding to $\lambda^{(1)}_1(\kappa,\rho)$. 
However, this is no more the case as soon as $n \ge 2$. In \cite{schwol10}, Schnitzler 
and Wolff considered $\cB$ as a perturbation of $\sum_{j=1}^n \delta_0(z_{j})$, leading 
to the results for $n \ge 2$. Expressing $\lambda^{(n)}_{p}(\kappa, \rho)$ in terms of the  
process $(Z(t))_{t \geq 0}$ does not seem very fruitful in  cases other than the one treated 
in \cite{garhey06} and\cite{schwol10}.  Therefore, it appeared to us more natural and more 
tractable to express $\lambda^{(n)}_{p}(\kappa, \rho)$ in terms of the process
$(X^{\kappa}_1(t), \cdots, X^{\kappa}_p(t), Y^{\rho}_1(t), \cdots,Y^{\rho}_n(t))$. 
We complete the intermittent picture by the following conjecture:
\begin{conjecture}[Intermittency]
Fix $n\in\N$. Then (see Fig.~{\rm\ref{fig-lambda01}--\ref{fig-lambda02}}),\\
(i) for $1\leq d\leq 2$, the system is full intermittent (proved for $n=1,2$); \\
(ii) for $d\geq 3$, the intermittency vanishes as  $\kappa$ increases. 
More precisely, for $d\geq 3$, there are three different regimes:
\begin{itemize}
\item[A:]{for $\kappa\in[0,\kappa_{2}^{(n)})$, the system is full intermittent (proved in a small 
neighborhood of $0$);}
\item[B:]{for $\kappa\in[\kappa_{2}^{(n)},nG_d(0))$, there exists 
$p=p(\kappa)\geq 3$ such that the system is $q$-intermittent for all $q\geq p$;}
\item[C:]{for $\kappa\in[nG_d(0),\infty)$, the system is not $p$-intermittent 
for any $p\geq 2$}.
\end{itemize}
\end{conjecture}

  To complete Theorem \ref{th4}, we close with a conjecture about critical 
$\kappa$'s, whose analogue for white noise potential was conjectured in Carmona 
and Molchanov \cite{carmol94} and partially proved in Greven and den Hollander 
\cite{grehol07}:
\begin{conjecture}[Critical $\kappa$'s]
 For all fixed $n\geq 1$ and $d$ large enough the $\kappa_p^{(n)}$'s are distinct 
 (see Fig.~{\rm\ref{fig-lambda02}}).
\end{conjecture}


\section{Proof of Theorem \ref{th1}}
\label{S2}

{\bf Step 1:} We first prove that if the limit in \refeq{aLyapdef} exists for $x=0$, 
then it exists for all $x\in \Z^d$ and does not depend on $x$ as soon as 
$(\kappa,\rho) \ne (0,0)$. To this end, let us introduce some notations. For any 
$t >0$, we denote 
\[ Y_t = (Y^{\rho}_1(t), \cdots, Y^{\rho}_n(t)) \in \Z^{dn}
\, \, ,  \,\, 
X_t = (X^{\kappa}_1(t), \cdots, X^{\kappa}_p(t)) \in \Z^{dp} \, .
\]
For $(x,y) \in \Z^{dp} \times \Z^{dn}$, $\EE^{X,Y}_{x,y}$ denote 
the expectation under the law of $(X_t,Y_t)_{t \ge 0}$ starting from $(x,y)$.
The same notation is used for $x \in \Z^d$ and $y \in \Z^d$. In that case,
it means that $X_0=(x, \cdots, x)$, $Y_0=(y,\cdots,y)$ and 
$\EE^{X,Y}_{x,y} = \EE_{y}^{\otimes n}\otimes \ES_{x}^{\otimes p}$. Finally,  
for  $x=(x_1,\cdots,x_p)\in \Z^{dp}$ and $y=(y_1,\cdots,y_n) \in \Z^{dn}$, set
\begin{equation}
\label{Idef}
I_p(x,y) = \sum_{j=1}^p \sum_{k=1}^n \delta_0(x_j-y_k) \, . 
\end{equation} 
Then, by time reversal for $Y$ in \refeq{lyapdef2}, 
for all $x \in \Z^d$ and  $t > 0$,
\begin{equation}
\label{momentp}
\EE_0^{\otimes n} \cro{u(x,t)^p}
= \sum_{z \in \Z^{dn}} \EE^{X,Y}_{x,z} \cro{\exp\pare{\int_0^t 
I_p(X_s,Y_s) \, ds} \delta_0(Y_t)} \, .
\end{equation} 
Using the Markov property at time 1 and the fact that 
$1 \le \exp \pare{\int_0^1 I_p(X_s,Y_s) \,ds}$, 
we get for $x_1$ and $x_2$ any fixed points in $\Z^d$,  
\begin{eqnarray*} 
\EE_0^{\otimes n} \cro{u(x_1,t)^p}
& \ge & \sum_{z \in \Z^{dn}} \EE^{X,Y}_{x_1,z} 
\cro{\delta_{(x_2, \cdots, x_2)}(X_1) \delta_z(Y_1) 
\exp\pare{ \int_1^t I_p(X_s,Y_s) \, ds }\delta_{0}(Y_{t})}
\\
& = & (p_1^{\kappa}(x_1,x_2))^p (p_1^{\rho}(0,0))^n
 \EE_0^{\otimes n} \pare{\cro{u(x_2,t-1)}^p} \, ,
\end{eqnarray*} 
where $p_t^{\nu}$ is the transition kernel of a simple random walk on $\Z^d$ with 
step rate $2d\nu$. This proves the independence of $\lambda_p$ w.r.t.\  $x$ as 
soon as $\kappa >0$, since in this case for all $x_1, x_2\in\Z^d$, 
$p_1^{\kappa}(x_1,x_2) >0$.

For $\kappa =0$, since the $X$-particles do not move, we have 
\begin{equation} 
\label{kappa0}
 \EE_0^{\otimes n} \cro{u(x_1,t)^p} = \EE_0\cro{\exp\pare{p \int_0^t
\delta_{x_1}(Y^{\rho}_1(s)) \, ds}}^n \, .
\end{equation} 
The same reasoning leads now to 
\[ \EE_0^{\otimes n} \cro{u(x_1,t)^p}
\ge p_1^{\rho}(0,x_1-x_2)^n 
\EE_0^{\otimes n} \pare{\cro{u(x_2,t-1)}^p} \, .
\]
{\bf Step 2: Variational representation.}
From now on, we restrict our attention
to the case $x=0$. The aim of this step is to give a variational representation of 
$\lambda_p^{(n)}(\kappa, \rho)$. To this end, we introduce further notations. Let
$(e_1, \cdots, e_d)$ be the canonical basis of $\R^d$. For $x = (x_1,\cdots,x_p) 
\in \Z^{dp}$, and $f: (x,y) \in \Z^{dp} \times \Z^{dn} \mapsto \R$, we set
\[ \nabla_x f(x,y) = \pare{\nabla_{x_1} f(x,y), \cdots, \nabla_{x_p} f(x,y)}
\in \R^{dp} \, , 
\]
where for $j \in \acc{1,\cdots, p}$, and $i \in \acc{1,\cdots,d}$,
\[ \bra{\nabla_{x_j} f(x,y),e_i} 
= f(x_1,\cdots,x_j+e_i,\cdots, x_p,y) - f(x,y) \, .
\]
The same notation is used for the $y$-coordinates, so that $\nabla_y f(x,y) 
\in \R^{dn}$.
 We also define 
\begin{eqnarray*} 
\Delta_x f(x,y) & = & \sum_{j=1}^p \Delta_{x_j} f (x,y)
\\
& = &  \sum_{j=1}^p \sum_{{z_j \in\Z^d} \atop {z_j\sim x_j}}
\big[f(x_1,\cdots,z_j,\cdots, x_p,y) - f(x_1,\cdots,x_j,\cdots, x_p,y)\big] \, .
\end{eqnarray*} 

\begin{proposition} 
Let $d\geq 1$ and $n,p\in\N$. For all $\kappa,\rho\in[0,\infty)$,
\begin{eqnarray}
\label{formvar} 
\lambda_p^{(n)}(\kappa, \rho) 
& = & \lim_{t \rightarrow \infty} \frac{1}{pt} 
\log \EE_0^{\otimes n} \cro{ u(0,t)^p } 
\nonumber
\\
& = & \frac{1}{p} 
			\sup_{
			 f \in l^2(\Z^{dp} \times \Z^{dn}) 
			\atop \nor{f}_2=1}
\hspace{-.2cm} \acc{ - \kappa \nor{\nabla_x f}_2^2 - \rho 
			\nor{\nabla_y f}_2^2 
+ \sum_{(x,y)}  I_p(x,y) f^2(x,y)} \, .\;\;\;
\end{eqnarray}  
\end{proposition} 

\begin{proof}
{\bf Upper bound.}
For a positive integer $m$, let $B^m_R$ denote the ball in $\Z^{dm}$ of radius 
$R = t \log(t)$ centered at the origin. We first prove the following lemma which 
states  we can restrict (\ref{momentp}) to  $X$ paths being in $B^p_R$ 
at time $t$ and $Y$ paths  starting from $B^n_R$.
\begin{lemma}\label{restrball} 
As $t\to \infty$,
\begin{equation}
\EE_0^{\otimes n}\cro{u(x,t)^p}=(1+o(1))\sum_{z \in B^n_R} \EE^{X,Y}_{0,z} 
\cro{\exp\pare{\int_0^t I_p(X_s,Y_s) \, ds}
\delta_0(Y_t) \ind{(B^p_R)}(X_t)}\, .
\end{equation}
\end{lemma}
\begin{proof}
It is enough to prove that 
\begin{equation}
r(t):=\frac{\displaystyle\EE_0^{\otimes n}\cro{u(x,t)^p}
-\sum_{z \in B^n_R}\EE^{X,Y}_{0,z} \cro{\exp\pare{\int_0^t I_p(X_s,Y_s) \, ds}
\delta_0(Y_t) \ind{(B^p_R)}(X_t)}}{\displaystyle\EE_0^{\otimes n}\cro{u(x,t)^p}}
\end{equation}
converges to $0$ as $t\to\infty$. Using the trivial bounds 
\begin{equation}
\label{Ibds}
1\leq \exp\pare{\int_0^t I_p(X_s,Y_s) \, ds}\leq \exp(tnp)
\end{equation} 
and  splitting the  sum in (\ref{momentp}), we get
\begin{align*}
r(t)&\leq \frac{e^{tnp}}
{\displaystyle\sum_{z \in Z^{dn}}\EE_{0,z}^{X,Y} \cro{ \delta_0(Y_t)}}
\Bigg(\sum_{z \notin B^n_R}\EE_{0,z}^{X,Y} \cro{ \delta_0(Y_t)}
+\sum_{z \in B^n_R}\EE^{X,Y}_{0,z} \cro{\delta_0(Y_t) \ind{(B^p_R)^c}(X_t)}\Bigg)\\
&\leq \frac{e^{tnp}}
{\displaystyle\sum_{z\in \Z^{dn}}\PP_z(Y_t=0)} 
\Bigg(\sum_{z \notin B^n_R}\PP_z(Y_t=0)
+\PP_0(X_t\notin B_R^p)\sum_{z\in B_R^n}\PP_z(Y_t=0)\Bigg)\\
&\leq e^{tnp}(\PP_0(Y_t\notin B_R^n)+ \PP_0(X_t\notin B_R^p)),
\end{align*}
where for the last two inequalities we used the time-reversal of $Y$. We have 
for $R=t\log(t)$ and large enough $t$
\begin{equation}
\label{BRbds}
\begin{aligned}
&\PP_0(Y_1^{\rho}(t)\notin B_R^1)\leq \exp[- C(d,\rho) t\log(t)],\\&
\PP_0(X_1^{\kappa}(t)\notin B_R^1)\leq \exp[- C(d,\kappa) t\log(t)]
\end{aligned}
\end{equation}
for some positive constants $C(d,\rho)$ and $C(d,\kappa)$ (see for instance Lemma 4.3 
in \cite{garmol90}). Using this we get
\begin{equation*}
r(t)\leq e^{tnp}\Big(n e^{-C(d,\rho)t\log t}+ p e^{-C(d,\kappa) t\log t}\Big)
\overset{t\to\infty}\longrightarrow 0.
\end{equation*}
This finishes the proof of the lemma.
\qed
\end{proof} 

Using Lemma \ref{restrball} it is enough to study the existence of 
\begin{eqnarray*} 
& & \lim_{t \rightarrow \infty} \frac{1}{t} \log  
\sum_{z \in B^n_R}  \EE^{X,Y}_{0,z} 
\cro{\exp\pare{\int_0^t I_p(X_s,Y_s) \, ds} \delta_0(Y_t) 
\ind{B^p_R}(X_t)}
\\
& & \qquad= \lim_{t \rightarrow \infty} \frac{1}{t} \log 
\bra{f_1, e^{\, t \cL_p} f_2}  \, ,
\end{eqnarray*} 
where $f_1: (x,y) \in \Z^{dp} \times \Z^{dn} 
\mapsto \delta_0(x) \ind{B^n_R}(y)$, $f_2: (x,y) \in \Z^{dp} \times \Z^{dn} 
\mapsto \ind{B^p_R}(x)  \delta_0(y)$, and $\cL_p$ is the bounded self-adjoint 
operator in $l^2(\Z^{dp} \times \Z^{dn})$ defined by 
\[ 
\cL_p f(x,y) = \kappa \Delta_x f(x,y) + \rho \Delta_y f (x,y) 
+ I_p(x,y) f(x,y),
\quad
(x,y) \in \Z^{dp} \times \Z^{dn} \, .
\]
For a linear operator $\mathcal{L}$ on $l^2(\Z^{dp} \times \Z^{dn})$ we define
$$
\nor{\mathcal{L}}_{2,2}:=\sup_{f \in l^2(\Z^{dp} \times \Z^{dn}) 
				\atop  \nor{f}_2=1}
			        \bra{f,\cL f}.
$$
Note that we have 
$$
\bra{f_1, e^{t\cL_p} f_2} \le \nor{f_1}_2 \nor{e^{t \cL_p}}_{2,2}
\nor{f_2}_2 = C(d,n,p) R^{d(n+p)/2}  \nor{e^{t \cL_p}}_{2,2},
$$ 
for some constant $C(d,n,p)>0$. Thus,
\[ \lim_{t \rightarrow \infty}
\frac{1}{t} \log \bra{f_1, e^{\, t \cL_p} f_2}
\le \nor{\cL_p}_{2,2} = \sup_{f \in l^2(\Z^{dp} \times \Z^{dn}) 
				\atop  \nor{f}_2=1}
			        \bra{f,\cL_p f} 
\, , 
\] 
which is the upper bound in \refeq{formvar}.	

\vspace{.3cm}
\noindent 
{\bf Lower bound.}
By \refeq{momentp} with $x=0$, it follows that
\begin{eqnarray*} 
\EE_0^{\otimes n} \cro{u(0,t)^p} 
& \ge & \EE^{X,Y}_{0,0} \cro{\exp\pare{\int_0^{t}
 I_p(X_s,Y_s) \, ds} \delta_0(X_t) \delta_0(Y_t)}
\\
& =  & \bra{\delta_0 \otimes \delta_0, e^{t \cL_p} (\delta_0 \otimes \delta_0)}
= \nor{e^{\frac{t}{2} \cL_p} (\delta_0 \otimes \delta_0)}_2^2
\\
& = & \sum_{x\in\Z^{dp}} \sum_{ y\in\Z^{dn}} \pare{e^{\frac{t}{2} \cL_p} 
 (\delta_0 \otimes \delta_0)(x,y)}^2 \, .
\end{eqnarray*} 
Restricting the sum over $B^p_{R}\times B^n_{R}$, and applying Jensen's 
inequality, we get
\begin{eqnarray*} 
& & \EE_0^{\otimes n} \cro{u(0,t)^p}
\\ 
& & \qquad \ge  \sum_{x \in B^p_R} \sum_{y \in B^n_R} 
\pare{e^{\frac{t}{2} \cL_p} 
 (\delta_0 \otimes \delta_0)(x,y)}^2
\\
& & \qquad \ge \frac{1}{|B^n_R|} \frac{1}{|B^p_R|} 
\pare{\sum_{x \in B^p_R} \sum_{y \in B^n_R} e^{\frac{t}{2} \cL_p} 
 (\delta_0 \otimes \delta_0)(x,y)}^2 
\\
& & \qquad = \frac{C(d,n,p)}{R^{d(n+p)}} 
\pare{ \sum_{x \in B^p_R} \sum_{y \in B^n_R} 
\EE^{X,Y}_{x,y} \cro{\exp\pare{\int_0^{t/2} I_p(X_s,Y_s) \, ds} 
\delta_0(X_{t/2})
\delta_0(Y_{t/2})}}^2
\\
& & \qquad = 
 \frac{C(d,n,p)}{R^{d(n+p)}} 
\pare{ \EE^{X,Y}_{0,0} \cro{\exp\pare{\int_0^{t/2} I_p(X_s,Y_s) \, ds} 
\ind{B^p_R}(X_{t/2}) \ind{B^n_R}(Y_{t/2})} }^2
\, .
\end{eqnarray*} 
Taking $R=t \log(t)$, we obtain that
\begin{eqnarray*}   
&& \liminf_{t \rightarrow \infty}   
\frac{1}{t} \log \EE_0^{\otimes n} \cro{u(0,t)^p}  
\\
&&\qquad \ge  \liminf_{t \rightarrow \infty} \frac{2}{t} 
\log \EE^{X,Y}_{0,0} \cro{ \exp \pare{\int_0^{t/2} I_p(X_s,Y_s) \, ds}
\ind{B^p_R}(X_{t/2}) \ind{B^n_R}(Y_{t/2}) }
\, .
\end{eqnarray*} 
On the other hand, by (\ref{Ibds}), (\ref{BRbds}) and our choice of $R$, we have
\begin{eqnarray*}   
&& \EE^{X,Y}_{0,0} \cro{ \exp \pare{\int_0^{t/2} I_p(X_s,Y_s) \, ds} 
\ind{(B^p_R \times B^n_R)^c}(X_{t/2},Y_{t/2})} 
\\
&& \qquad \le \exp\Big(\frac{tnp}{2}\Big) 
\PP_{0}(X_{t/2}\notin B_{R}^p)\PP_{0}(Y_{t/2}\notin B_{R}^n)
\\
&& \qquad \le np\exp\Big[\frac{tnp}{2}  - \big(C(d,\rho)+C(d,\kappa)\big)t \log(t)\Big] \, , 
\end{eqnarray*}
and therefore, with a similary reasoning as in the proof of Lemma \ref{restrball} we get 
\[ 
\liminf_{t \rightarrow \infty} 
\frac{1}{t} \log \EE_0^{\otimes n} \cro{u(0,t)^p} 
\ge  \liminf_{t \rightarrow \infty} \frac{2}{t}  
\log \EE^{X,Y}_{0,0} \cro{ \exp \pare{\int_0^{t/2} I_p(X_s,Y_s) \, ds}}
\, .
\]
  Now, the occupation measure $\frac{1}{t} \int_0^t \delta_{(X_s,Y_s)} \, ds$
satisfies a weak large deviations principle (LDP) 
in the space $\cM_1(\Z^{dp}\times 
\Z^{dn})$ of probability measures on $\Z^{dp}\times \Z^{dn}$, endowed with
the weak topology. The speed of this LDP is $t$ and the rate function is given
for all $\nu \in \cM_1(\Z^{dp}\times \Z^{dn})$ by 
\[ J(\nu) = \kappa \nor{\nabla_x \sqrt{\nu}}_2^2 + \rho \nor{\nabla_y
\sqrt{\nu}}_2^2 \, ,
\] 
(see e.g.\ den Hollander \cite{hol00}, Section IV.4). Since $I$ is bounded, the lower 
bound in Varadhan's integral lemma (see e.g.\ den Hollander \cite{hol00}, Section III.3) 
yields 
\[
\liminf_{t \to \infty} \frac{1}{t} \log \EE_0^{\otimes n} \cro{u(0,t)^p} 
\ge \sup_{\nu \in  \cM_1(\Z^{dp}\times \Z^{dn})}
\acc{ \sum_{(x,y)} I_p(x,y) \nu(x,y) - J(\nu) } \, .
\] 
Setting $f(x,y) = \sqrt{\nu}(x,y)$ gives
then the lower bound in \refeq{formvar}.
\qed
\end{proof}

\vspace{.2cm}
\noindent
{\bf Step 3: Properties of $\lambda^{(n)}_p$.} Since $0 \le I_p(x,y) \le np$,
we clearly have $0 \le \lambda^{(n)}_p \le n$. Using representation
\refeq{formvar}, we can conclude that the function $(\kappa,\rho) \mapsto 
\lambda^{(n)}_p(\kappa,\rho)$ is convex and non-increasing in 
$\kappa$ and $\rho$. Moreover, $\lambda^{(n)}_p(\kappa,\rho)$ is lower 
semi-continuous since it is supremum of functions that are linear in  $\kappa$ 
and $\rho$. Finally, since every finite convex function is also upper 
semi-continuous, $\lambda^{(n)}_p$ is upper semi-continuous. Hence, 
$\lambda^{(n)}_p(\kappa,\rho)$ is continuous.


\section{Proof of Theorems \ref{th2}--\ref{th3}}
\label{S3}
By symmetry, note that for all $n, p\in\N$ and $\kappa,\rho\in[0,\infty)$, 
\begin{equation}
\label{sym}
\lambda^{(n)}_p(\kappa,\rho) = \frac{n}{p} \lambda^{(p)}_n(\rho,\kappa) \, .
\end{equation}

\subsection{Proof of Theorem \ref{th2}}
\label{S3.1}
{\bf Proof of (i)}: By continuity, $\lim_{\kappa \rightarrow 0}
\lambda^{(n)}_p(\kappa,\rho) = \lambda^{(n)}_p(0,\rho)$.  Now for 
$\kappa = 0$, the $X$ particles do not move so that 
$\EE_0^{\otimes n} \cro{u(0,t)^p} = \EE_0 \pare{\exp \pare{ p L^Y_t(0)}}^n$ 
(see \refeq{kappa0}), where  $L^Y_t(0)$ is the local time at $0$ of a simple 
random walk in $\Z^d$  with rate $2d\rho$. Using the LDP for $L^Y_t$, 
we obtain
\[  
\lambda^{(n)}_p(0,\rho) = \frac{n}{p} \sup_{ f \in l^2(\Z^d) 
						\atop  \nor{f}_2=1}
 \bra{f, (\rho \Delta + p \delta_0)f} 
= n \mu(\rho/p) \, . 
\]

\vspace{.3cm}
\noindent
{\bf Proof of (ii)}: For all $n, p \in \N$ and $\kappa, \rho\in[0,\infty)$, we have
\[
\lambda^{(n)}_p(\kappa,\rho) \ge \lambda^{(n)}_1(\kappa,\rho) 
= n \lambda^{(1)}_n(\rho,\kappa) \ge n \lambda^{(1)}_1(\rho,\kappa) 
= n \mu(\kappa + \rho)
\, ,
\] 
where the last equality is proved in \cite{garhey06} and comes from the
fact that $X^1_t-Y^1_t$ is a simple random walk in $\Z^d$ with jump 
rate $2d(\kappa + \rho)$. Since $G_d(0) = \infty$ for $d=1,2$, it follows from
\refeq{muprop} 
that $ \lambda^{(n)}_p(\kappa,\rho) > 0$ for $d=1,2$.

Let us prove that $\lim_{\kappa \rightarrow \infty}  
\lambda^{(n)}_p(\kappa,\rho) = 0$. By monotonicity in $\rho$,
\begin{equation}
\label{ub1}
 \lambda^{(n)}_p(\kappa,\rho) \le  \lambda^{(n)}_p(\kappa,0)
= n \mu(\kappa /n) \, .
\end{equation}
Hence the only thing to prove is that $\lim_{\kappa \rightarrow \infty}
\mu(\kappa) = 0$.
 To this end, one can use the discrete Gagliardo-Nirenberg
 inequality: there exists a constant $C$ such that for all $f: \Z^d \mapsto \R$,
 \begin{eqnarray} 
 \label{GNineq1}
 \mbox{ for } d=1 \, ,  &  & \nor{f}_{\infty}^2 \le  C \nor{f}_2 
			\nor{\nabla f}_2 \, ; \\
\label{GNineq2}		
 \mbox{ for } d=2 \, , & & \nor{f}_{4}^2 \le C \nor{f}_2 \nor{\nabla f}_2 \, .
 \end{eqnarray}
The proof of these inequalities follows the same lines as the proof of the usual 
Gagliardo-Nirenberg inequality (see e.g.\ Brezis \cite{brezis}). For completeness 
a short proof is given in the appendix.   
From (\ref{GNineq1}) and (\ref{GNineq2}), we get  for all $f \in l_2(\Z^d)$ with $\|f\|_{2}=1$,
\begin{eqnarray*} 
 - \kappa \nor{\nabla f}_2^2 + f(0)^2 
& \le &  
 	\left\{ 
 	\begin{array}{ll} 
  	- \kappa \nor{\nabla f}_{2}^2 + \nor{f}_{\infty}^2 
 	& \mbox{ for } d=1 
 	\\
  	- \kappa  \nor{\nabla f}_{2}^2 + \nor{f}_4^2
  	& \mbox{ for } d=2
 	\end{array} 
 	\right.
\\
& \le & - \kappa \nor{\nabla f}_2^2 + C   \nor{\nabla f}_2 \, .
\end{eqnarray*} 
Taking the supremum over $f$ yields 
\[ \mu(\kappa) 
\le \sup_{x \ge 0}   
\pare{ - \kappa x^2 + Cx } = \frac{C^2}{4 \kappa} \, .
\]
The strict monotonicity is now an easy consequence of the fact 
that $\kappa \mapsto \lambda^{(n)}_p(\kappa,\rho)$ is convex, positive, 
non increasing, and tends to $0$ as $\kappa\to\infty$.

\vspace{.3cm}
\noindent
{\bf Proof of (iii)}:  By \refeq{sym} and \refeq{ub1},  we get
\begin{equation}
\label{ub2}
\lambda^{(n)}_p(\kappa,\rho) 
\le n \min \pare{\mu(\kappa/n), \mu(\rho/p)}.
\end{equation}
Then the claim follows by (\ref{muprop}).

\subsection{Proof of Theorem \ref{th3}}
\label{S3.2}

\noindent
{\bf Proof of (i)}:  Fix $\epsilon > 0$. Let $f$ approaching the supremum 
in the variational representation \refeq{formvar} of $\lambda_p^{(n)}(\kappa,0)$,  so that 

\begin{eqnarray*}
p \lambda_p^{(n)}(\kappa,0) - \epsilon
& \le &
- \kappa \nor{\nabla_x f}_2^2 + \sum_{x\in\Z^{dp}} \sum_{y\in\Z^{dn}} I_p(x,y) f^2(x,y)
\\
& \le  &
p \lambda_p^{(n)}(\kappa,\rho) + \rho \sup_{f \in l^2(\Z^{dp} \times \Z^{dn})
					\atop  \nor{f}_2 =1}
\nor{\nabla_y f}_2^2  \, .
\end{eqnarray*}
For $x \in \Z^{dp}$, set $f_x: y \in \Z^{dn} \mapsto f(x,y)$. 
Since the bottom of the spectrum of $\Delta$ in $l^2(\Z^{dn})$
is $-4dn$, 
\[ \sum_{y\in\Z^{dn}} \nor{\nabla_y f_x  (y)}_2^2 \le 4dn 
\sum_{y\in\Z^{dn}} f^2_x(y)\, ,
\] 
for all $x \in \Z^{dp}$. Hence,
\[ \sum_{x\in\Z^{dp}} \sum_{y\in\Z^{dn}} \nor{\nabla_y f_x  (y)}_2^2 \le 4dn 
\sum_{x\in\Z^{dp}} \sum_{y\in\Z^{dn}} f^2_x(y) = 4dn \, .
\]
Therefore, for all $\epsilon > 0$, 
\[ 
p \lambda_p^{(n)}(\kappa,0) - \epsilon
\le p \lambda_p^{(n)}(\kappa,\rho) + 4dn  \rho \, .
\]
Letting $ \epsilon \to 0$ yields, 
 \begin{equation}
 \label{lambdabds}
\lambda_p^{(n)}(\kappa,0)  - \frac{4dn\rho}{p}
 \le \lambda_p^{(n)}(\kappa,\rho) \le  \lambda_p^{(n)}(\kappa,0)  \, ,
\end{equation}
which, after letting $p\to\infty$, gives the claim.
 
\vspace{.3cm} 
\noindent
{\bf Proof of (ii)}: By \refeq{sym}, $\lim_{n \to \infty} \lambda^{(n)}_p(\kappa,\rho)
= \lim_{n \to \infty} \frac{n}{p} \lambda^{(p)}_n(\rho, \kappa)$ and by (i), 
\[ \lim_{n \to \infty} \lambda^{(p)}_n(\rho, \kappa)
\geq \lambda^{(p)}_n(\rho, 0) = p \mu(\rho/p) > 0 \, , \mbox{ for } p > \rho/G_d(0)
\, .
\]
Hence, for $p > \rho/G_d(0)$, $\lim_{n \to \infty} \lambda^{(n)}_p(\kappa,\rho) = + \infty$.

\vspace{.3cm}
\noindent
{\bf Proof of (iii)}: This is a direct consequence of Theorem \ref{th2}(iii).
 

\section{Proof of Theorem \ref{th4}}
\label{S4}
{\bf Proof of (i)}: We first prove that 
\begin{equation}
\label{formvarkappa}
\kappa^{(n)}_p(\rho)
= \sup_{ f \in l_2(\Z^{dp}\times\Z^{dn}) \atop 	\nor{f}_2=1}
\frac{\sum_{x,y} I_p(x,y) f^2(x,y) - \rho \nor{\nabla_y f}^2_2}
{\nor{\nabla_x f}^2_2}  \, ,
\end{equation}
with $I$ defined as in \eqref{Idef}.
Indeed, let us denote by $S$ the supremum in the right-hand 
side of \refeq{formvarkappa}. 

If $\kappa \ge \kappa^{(n)}_p(\rho)$, then $\lambda^{(n)}_p(\kappa, \rho)
= 0$. Therefore, using \refeq{formvar}, for all 
$f \in l_2(\Z^{dp}\times\Z^{dn})$ such that $\nor{f}_2=1$, 
\[ 
\sum_{x\in\Z^{dp}} \sum_{y\in\Z^{dn}} I_p(x,y) f^2(x,y) 
- \rho \nor{\nabla_y f}^2_2 
\le \kappa \nor{\nabla_x f}^2_2 , 
\]
so that $\kappa \ge S$. Hence $\kappa^{(n)}_p(\rho) \ge S$.
On the opposite direction, we can assume that $S < \infty$. 
Then, by definition of $S$, for all 
$f \in l_2(\Z^{dp}\times\Z^{dn})$ such that $\nor{f}_2=1$, 
\[ \sum_{x\in\Z^{dp}} \sum_{y\in\Z^{dn}} I_p(x,y) f^2(x,y) 
- \rho \nor{\nabla_y f}^2_2 
\le S \nor{\nabla_x f}^2_2 \, .
\]
Thus, for all 
$f \in l_2(\Z^{dp}\times\Z^{dn})$ such that $\nor{f}_2=1$, and all
$\kappa \ge S$, 
\[ \sum_{x\in\Z^{dp}} \sum_{y\in\Z^{dn}} I_p(x,y) f^2(x,y) 
- \rho \nor{\nabla_y f}^2_2 - \kappa 
\nor{\nabla_x f}^2_2 \le (S-\kappa)  \nor{\nabla_x f}^2_2 \le 0 \, .
\]
Hence, for all $\kappa \ge S$, $\lambda^{(n)}_p(\kappa, \rho)
= 0$, i.e.,  $\kappa \ge \kappa^{(n)}_p(\rho)$. Hence,
$S \ge \kappa^{(n)}_p(\rho)$. This proves \refeq{formvarkappa}.
 
Since $\rho \mapsto \kappa^{(n)}_p(\rho)$ is a supremum of 
 linear functions, it is lower semi-continuous and convex.  
It is also obvious that $\rho \mapsto \kappa^{(n)}_p(\rho)$ is non increasing. The 
continuity follows then from the finiteness of $\kappa^{(n)}_p(\rho)$.
 
The lower bound in \refeq{kappabds}  is a direct consequence of \refeq{lambdabds}. 
Indeed, since $\lambda^{(n)}_p(\kappa,0)$ $= n \mu(\kappa/n)$, it follows from 
\refeq{lambdabds} that if 
$\mu(\kappa/n) > 4d\rho/p$, then $\kappa < \kappa^{(n)}_p(\rho)$. This yields the bound:  
\[ \kappa^{(n)}_p(\rho) \ge n \mu^{-1}(4d \rho/p) \, .
\]
Using the symmetry relation \refeq{sym}, we also get from \refeq{lambdabds}
that  
\[ \lambda^{(n)}_p(\kappa,\rho) \ge n \mu(\rho/p)-4d\kappa \, .
\]
This leads to $\kappa^{(n)}_p(\rho) \ge \frac{n}{4d} \mu(\rho/p)$.
Hence, if $\rho/p < G_d(0)$, $\kappa^{(n)}_p(\rho) > 0$. We have
already seen that $\kappa^{(n)}_p(\rho) = 0$ if $\rho/p \ge G_d(0)$.
Since $\lambda^{(n)}_p(\kappa,0)=n \mu(\kappa/n)$, it follows that
$\kappa^{(n)}_p(0) = nG_d(0)$. Using convexity, we have, for all $\rho 
\in [0,pG_d(0)]$,
\[ \kappa^{(n)}_p(\rho) \le
\frac{\kappa^{(n)}_p(pG_d(0))-\kappa^{(n)}_p(0)}{pG_d(0)} \rho 
+ \kappa^{(n)}_p(0)= n \pare{G_d(0) - \rho/p} \, .
\]
Since $\kappa^{(n)}_p(\rho) = 0$ if $\rho/p \ge G_d(0)$, then the upper
bound in \refeq{kappabds} is proved. 
 
\vspace{.3cm}
\noindent 
{\bf Proof of (ii)}: To prove \refeq{kappabds2}, let $f_0$ be the function 
\[ f_0 (x,y) = \prod_{i=1}^p \frac{G_d(x_i)}{\nor{G_d}_2} \, \prod_{j=1}^n \delta_0(y_j) \, .
\]
Note that for $d \ge 5$, $\nor{G_d}_2 < \infty$, so that $f_0$ is well-defined, 
and has $l_2$-norm  equal to $1$. From \refeq{formvarkappa}, we get 
\[ \kappa^{(n)}_p(\rho) \ge \frac{\sum_{x,y} I_p(x,y) f_0^2(x,y) - \rho \nor{\nabla_y f_0}_2^2}
{ \nor{\nabla_x f_0}_2^2} \, .
\]
An easy computation then gives
\[  \sum_{x,y} I_p(x,y) f_0^2(x,y) = np 	\frac{G_d^2(0)}{\nor{G_d}^2_2} \, ,
\]
\[ \nor{\nabla_y f_0}_2^2 = n \nor{\nabla_{y_1} \delta_0}_2^2 = 2d n \, ,
\]
and
\[ \nor{\nabla_x f_0}_2^2 = p \frac{ \nor{\nabla_{x_1} G_d}_2 ^2}{\nor{G_d}_2^2}=
p \frac{G_d(0)}{\nor{G_d}_2^2} \, , 
\]
since $\nor{\nabla_{x_1} G_d}_2 ^2 = \bra{G_d , - \Delta G_d} =\bra{G_d , \delta_0}= G_d(0)$.
This gives \refeq{kappabds2}. 

\vspace{.3cm}
\noindent
{\bf Proof of (iii):}  The inequality \refeq{kappa_croit} is clear if $ \rho \in [(p-1)G_d(0), pG_d(0))$,
since in this case, $\kappa^{(n)}_{p-1}(\rho)=0 < \kappa^{(n)}_{p}(\rho)$. We assume therefore
that $\rho \in (0, (p-1)G_d(0))$. From \refeq{kappabds}, we have $\kappa^{(n)}_{p-1}(\rho) \le
n G_d(0) - \rho n/(p-1)$, whereas, from \refeq{kappabds2}, $\kappa^{(n)}_{p}(\rho) \ge
n G_d(0) - \rho n/(p \alpha_d)$. Hence $\kappa^{(n)}_{p-1}(\rho) < \kappa^{(n)}_{p}(\rho)$
as soon as $\alpha_d > \frac{p-1}{p}$.  This gives the claim.


\section{Proof of Corollary \ref{th5}}
\label{S5}
{\bf Proof of (i):} The function 
$p \mapsto \lambda^{(n)}_p(\kappa, \rho)$ increases from 
$\lambda^{(n)}_1(\kappa, \rho)$ to $n \mu(\kappa/n)$. Hence, there
exists $p$ such that 
$\lambda^{(n)}_p(\kappa, \rho) < \lambda^{(n)}_{p+1} (\kappa, \rho)$
as soon as  $\lambda^{(n)}_1(\kappa, \rho) < n \mu(\kappa/n)$. 
But $ n \mu(\kappa/n) = \lambda^{(n)}_1 (\kappa, 0)$. Hence,  if 
$\lambda^{(n)}_1(\kappa, \rho) = n \mu(\kappa/n)$, the convex 
decreasing function $\rho \mapsto  \lambda^{(n)}_1(\kappa, \rho)$ 
is constant. Being equal to 0 for $\rho \ge G_d(0)$, we get that 
$n \mu(\kappa/n) = 0$, which can not be  the case if $\kappa < nG_d(0)$.
This ends the proof of the first part.
 
If $\kappa \ge n G_d(0)$, then $\lambda^{(n)}_p(\kappa, \rho)=0$, for all 
$p\ge 1$, and the system is not intermittent. This proves the second part.

\vspace{.3cm}
\noindent
{\bf Proof of (ii):} For all $p \in \N\setminus\{1\}$ by Lemma \ref{lm} 
for $d$ large enough we have $\alpha_d > \frac{p-1}{p}$. This implies that 
$\alpha_d>\frac{q-1}{q}$ for all $q\in\N\setminus\{1\}$ and $q\leq p$.
Hence, by Theorem \ref{th4}(iii), for all $q\in\N\setminus\{1\}$ with $q\leq p $ 
we have $\kappa^{(n)}_{q-1}(\rho)< \kappa^{(n)}_{q}(\rho)$, for all 
$\rho \in (0, pG_d(0))$. Hence, in the domain 
$$\acc{(\kappa,\rho)\colon \rho \in (0, qG_d(0))\, , \,\,\, 
\kappa^{(n)}_{q-1}(\rho)\le \kappa < \kappa^{(n)}_{q}(\rho)}$$ one has 
\[ \lambda^{(n)}_1(\kappa, \rho)= \cdots = \lambda^{(n)}_{q-1}(\kappa, \rho)=0 
< \lambda^{(n)}_{q}(\kappa, \rho) \, ,
\]
which proves the desired result.

\begin{acknowledgement}
The research in this paper was supported by the ANR-project MEMEMO.
\end{acknowledgement}


\section*{Appendix: Proof of lemma \ref{alpha}}
\addcontentsline{toc}{section}{Appendix}

For a function $f: \Z^d \mapsto \R$, let $\hat{f}$ denote the Fourier transform of $f$:
\[ \hat{f}(\theta) = \sum_{x \in \Z^d} e^{i \bra{\theta,x}} f(x) \quad \forall \theta \in [0,2\pi]^d \, .
\]
Then, the inverse Fourier transform is given by 
\[ f(x) = \frac{1}{(2\pi)^d} \int_{[0,2\pi]^d} \, e^{-i \bra{\theta,x}} \hat{f}(\theta) \,  d \theta \, , 
\]
and the Plancherel's formula reads
\[ \sum_{x \in \Z^d} f^2(x)= \frac{1}{(2 \pi)^d} \int_{[0,2\pi]^d} |\hat{f}(\theta)|^2 \, d\theta \, .
\]
Using the equation $\Delta G_d = -\delta_0$ we get that 
\[ \hat{G}_d(\theta) = \frac{1}{2 \sum_{i=1}^d (1 - \cos(\theta_i))} \, .
\]
Hence, 
 \begin{align*}
G_d(0) & =  \frac{1}{(2\pi)^d} \int_{[0,2\pi]^d} \frac{d\theta}{2 \sum_{i=1}^d (1 - \cos(\theta_i))} \\
       & =  \;\;\frac{1}{\pi^d}\;\; \int_{[0,\pi]^d} \frac{d\theta}{2 \sum_{i=1}^d (1 - \cos(\theta_i))} \\
       & = \;\EE \cro{ \frac{1}{2 \sum_{i=1}^d (1-\cos(\Theta_i))}} \, 
 \end{align*} 
where the random variables $(\Theta_i)$ 
are i.i.d.\ with uniform distribution on $[0,\pi]$.
 Moreover, by Plancherel's formula we have 
 \[ \nor{G_d}_2^2 = \frac{1}{(2 \pi)^d} 
 \int_{[0,2\pi]^d} \frac{d\theta}{\pare{2 \sum_{i=1}^d (1 - \cos(\theta_i))}^2} 
 = \EE \cro{ \frac{1}
{\pare{2 \sum_{i=1}^d (1-\cos(\Theta_i))}^2}} \,.
 \]
 Thus, 
 \[ \alpha_d = \frac{G_d(0)}{2d \nor{G_d}_2^2} = \frac{\EE \cro{\frac{1}{\bar{S}_d}}}
 				{\EE \cro{ \frac{1}{\bar{S}_d^2}}} \, , 
\]
where $  \bar{S}_d = \frac{1}{d} \sum_{i=1}^d (1-\cos(\Theta_i)) $. 
Applying H\"older's and Jensen's inequality, we get that 
\[ \alpha_d \le \frac{1}{\sqrt{\EE \cro{ \frac{1}{\bar{S}_d^2}}}} \le
\EE(\bar{S}_d)
=1 \, .
\]
By the law of large numbers, $\bar{S}_d$ converges almost surely to $\EE\cro{1-\cos(\Theta)}=1$
as $d$ tends to infinity. We are now going to prove that $\bar{S}_d^{-2}$ is uniformly integrable
by showing that for all $p > 2$,
\begin{equation}
\label{UI}
\sup_{d > 2p} \EE \cro{\bar{S}_d^{-p}} < \infty.
\end{equation}
Indeed, let $\epsilon \in (0,\pi)$ be a small positive number to be fixed later.
Let 
\[ \cI = \acc{i \in \{1,\cdots,d\}\colon 0 \le \Theta_i \leq \epsilon} \, .
\]
\[ \bar{S}_d \ge \frac{1}{d} \sum_{i \notin \cI} (1-\cos(\epsilon)) 
		+ \frac{c_{\epsilon}}{d} \sum_{i \in \cI} \Theta_i^2
\, ,
\]
where $c_{\epsilon} = \inf_{0 \le \theta \le \epsilon} \frac{1-\cos(\theta)}
{\theta^2} \rightarrow 1/2$ when $\epsilon \rightarrow 0$. Therefore, 
\[
\EE \cro{\bar{S}_d^{-p}} 
\le  d^{p}
	\sum_{k=0}^d \sum_{{I \subset \acc{1,\cdots, d}} \atop {|I| = k}}
        \EE \cro{
	\frac{ \ind{\cI = I} }{\pare{(1-\cos(\epsilon))(d-k) + c_{\epsilon}
	\sum_{i \in  I} \Theta_i^2}^p}}
\, .
\]
Since the last expectation only depends on $|I|$, we get
\[
\EE \cro{\bar{S}_d^{-p}} 
\le d^{p} \sum_{k=0}^d \pare{\begin{array}{c} d \\ k \end{array}}
		a(k,\epsilon,d)
\, ,
\] 
with
\[ a(k,\epsilon,d) := \frac{1}{\pi^d} 
\int_{{0 \le \theta_1, \cdots , \theta_k \le \epsilon} \atop 
{\epsilon  \le \theta_{k+1},
\cdots, \theta_d \le \pi}} \frac{ d\theta_1 \cdots d\theta_d}
{\pare{(1-\cos(\epsilon))(d-k) 
+ c_{\epsilon} (\theta_1^2 + \cdots + \theta_k^2)}^p}
\, .
\]
Let $\omega_d$ denote the volume of the $d$-dimensional unit
ball. For $k=d$,
\begin{eqnarray*}  
a(d,\epsilon,d) & = & \frac{1}{\pi^d} 
 \int_{0 \le \theta_1, \cdots , \theta_d \le \epsilon} 
\frac{ d\theta_1 \cdots d\theta_d}
{ c_{\epsilon}^p \nor{\theta}^{2p}}
\\
& \le & \frac{1}{c_{\epsilon}^p \pi^d} \omega_d \int_0^{\sqrt{d} \epsilon}
	r^{d-2p-1} dr 
\\
& = &  \pare{\frac{\epsilon}{\pi}}^d \frac{1}{(c_{\epsilon} \epsilon^2)^p}
d^{\frac{d}{2}-p} \frac{\omega_d}{d - 2p} \, ,  
\end{eqnarray*} 
for $d > 2p$. 

 Note that for large $d$, $\omega_d \simeq \frac{(2e\pi)^{d/2}}{\sqrt{\pi d} d^{d/2}}$.
Therefore, as $d\to\infty$
\[ d^p \pare{d \atop d} a(d,\epsilon, d) = O \pare{ d^{-3/2} 
(\epsilon^2 2e/\pi)^{d/2}} \, .
\] 
 If $\epsilon$ is chosen so that $\epsilon^2 \le
\pi/(2e)$, we obtain that $\lim_{d \rightarrow \infty}  
d^p \pare{d \atop d} a(d,\epsilon , d) = 0$.

For $k \le d-1$, 
\[ a(k,\epsilon,d) \le \frac{1}{(1-\cos(\epsilon))^p}
\frac{1}{(d-k)^p} 
\pare{\frac \epsilon \pi}^k \pare{1- \frac \epsilon \pi}^{d-k} \, ,
\]
and  $d^p \pare{d \atop k} a(k,\epsilon, d) 
\le \frac{1}{(1-\cos(\epsilon))^p} 
\EE \cro{ \ind{N = k} (1-N/d)^{-p}}$, where $N$ is a Binomial random variable 
with parameters $d$ and $\epsilon/\pi$. Hence, for $\epsilon < \min(\pi, \sqrt{\pi/(2e)})$, 
\begin{eqnarray*} 
&&\EE\cro{\frac{1}{\bar{S}_d^{p}}} 
\nonumber
\\
&& \,\,\,\,\le  \frac{1}{(1-\cos(\epsilon))^p}
\EE \cro{ \ind{N \le d-1} (1-N/d)^{-p}}
+ O\pare{ d^{-3/2} }
\\
&& \,\,\,\,\le  \frac{d^p}{(1-\cos(\epsilon))^p} 
\PP \cro{d \frac{2 \epsilon}{\pi} \le N \le d-1}
     + \frac{1}{(1-\cos(\epsilon))^p  (1- \frac{2\epsilon}{\pi})^{p}}  
     + O\pare{ d^{-3/2} } \, .
\end{eqnarray*} 
Now, by the large deviations principle satisfied by $N/d$, there is an
$i(\epsilon) > 0$ such that $\PP \cro{ N \ge d 2 \epsilon /\pi}
\le \exp(-d i(\epsilon))$.
This ends the proof of \refeq{UI}.

  Using the uniform integrability \refeq{UI}, and the fact that $\bar{S}_d$
converges a.s. to 1, we obtain that $\EE\cro{\frac{1}{\bar{S}_d}}$ and
$\EE\cro{\frac{1}{\bar{S}_d^2}}$ both converge to 1, 
when $d$ goes to infinity.

\section*{Appendix: Proof of proposition \ref{eigenfunc}.}

Let $f\in l_2(\Z^{d(1+n)})$ with $\nor{f}_2=1$, such that 
$\cL_1 f = \lambda_1^{(n)}(\kappa, \rho) f$. Define 
\begin{equation*}
\tilde{f}(x_1, x_2, y)
= f(x_1,y) f(x_2,y),
\quad x_1,y_1\in\Z^d,\, y\in\Z^{dn}\, .
\end{equation*} 
Since
\[ \sum_{x_1, x_2, y} \tilde{f}^2(x_1, x_2, y)
= \sum_y \pare{\sum_x f^2(x,y)}^2 
\le \pare{\sup_y \sum_x f^2(x,y)} \nor{f}_2^2 \le \nor{f}_2^4
\, ,
\]
it follows that $\tilde{f}$ is in $l_2(\Z^{d(2+n)})$. A simple computation yields
\[ \Delta_{x_1} \tilde{f}(x_1, x_2, y) = f(x_2,y) \Delta_{x}f(x_1,y) 
\, ,   \quad 
\Delta_{x_2} \tilde{f}(x_1, x_2, y) = f(x_1,y) \Delta_{x}f(x_2,y) 
\, ,
\]
and
\begin{eqnarray*}
\Delta_{y} \tilde{f}(x_1, x_2, y) = f(x_2,y) \Delta_{y}f(x_1,y) 
			& + & f(x_1,y) \Delta_{y}f(x_2,y) 
			\\
			& + &  \sum_{z \sim y} (f(x_1,z)-f(x_1,y)) (f(x_2,z)-f(x_2,y)) \, .
\end{eqnarray*}
Since
\begin{equation*}
I_2 \tilde f(x_1,x_2,y)= f(x_{2},y)I_1 f(x_1,y) + f(x_{1},y) I_1 f(x_2,y)\, ,
\end{equation*} 
(recalling (\ref{Idef})), this leads to 
\begin{equation*}
\begin{aligned} 
&\cL_2 \tilde{f}(x_1,x_2,y)\\  
&\quad=   2  \lambda_1^{(n)}(\kappa, \rho) \tilde{f}(x_1,x_2,y) 
+ \rho  \sum_{z \sim y} \big(f(x_1,z)-f(x_1,y)\big) \big(f(x_2,z)-f(x_2,y)\big) 
\end{aligned}
\end{equation*}
(recalling (\ref{Lpdef}). Therefore 
\begin{eqnarray*}
 \lambda_2^{(n)}(\kappa, \rho) \nor{\tilde{f}}_2^2 
&  \ge & \frac{1}{2} \bra{\tilde{f}, \cL_2 \tilde{f} }
\\
&  = &  \lambda_1^{(n)}(\kappa, \rho) \nor{\tilde{f}}_2^2  +  \frac{\rho}{2}
 \sum_{y,z\sim y}  \pare{\sum_x f(x,y) (f(x,z)-f(x,y)) }^2 \, .
\end{eqnarray*}
Note that 
\begin{equation*}
\sum_{y,z\sim y}  \pare{\sum_x f(x,y) (f(x,z)-f(x,y)) }^2 \ge 0
\end{equation*} 
with equality to $0$ if and only if for all  $y$ and $z \sim y$, $\sum_x f(x,y) (f(x,z)-f(x,y)) = 0$. Interchanging 
the role of $z$ and $y$ yields $\sum_x  (f(x,z)-f(x,y))^2 = 0$, so that  for all $x$, $y$  and $z \sim y$, 
$f(x,z) = f(x,y)$. Hence, for all $x$, $y$,  $f(x,y)=f(x,0)$. This is impossible since $\nor{f}_2=1$.
Thus, $\lambda_2^{(n)}(\kappa, \rho) > \lambda_1^{(n)}(\kappa, \rho)$.
\qed

\section*{Appendix: Proof of the discrete Gagliardo-Nirenberg inequality.}
\addcontentsline{toc}{section}{Appendix}
{\bf Proof for $d=1$.}
One can assume that $\nor{f}_2 < \infty$, otherwise there is nothing to prove.  
Hence $\lim_{\va{x} \rightarrow \infty} \va{f(x)} = 0$, and by the Cauchy-Schwarz inequality,
we have for all $x \in \Z$, 
\begin{eqnarray*}
f^2(x) & = &  \sum_{j=-\infty}^{x} f^2(j) - f^2(j-1)
\\ 
& \le & \sum_{j=-\infty}^{+\infty} \va{f(j) - f(j-1)} (\va{f(j)}+\va{f(j-1)})
\\
& \le & 2 \sqrt{\sum_{j} \va{f(j) - f(j-1)}^2} \sqrt{ \sum_{j} f^2(j)}
\\
& = & 2 \nor{f }_2 \nor{\nabla f}_2 \, ,
\end{eqnarray*}
which proves \refeq{GNineq1} with $C=2$. 
\\
{\bf Proof for $d=2$.}
Here again, one can assume that $\nor{f}_2 < \infty$,  and consequently 
$\lim_{\va{x_1} \rightarrow \infty} \va{f(x_1,x_2)} = 0$.  Then, by the Cauchy-Schwarz 
inequality, we have for all $x_1, x_2 \in \Z$, 
\begin{eqnarray*}
f^2(x_1,x_2) & = &  \sum_{j_1=-\infty}^{x_1} f^2(j_1,x_2) - f^2(j_1-1,x_2)
\\ 
& \le & \sum_{j_1=-\infty}^{+\infty} \va{f(j_1,x_2) - f(j_1-1,x_2)} (\va{f(j_1,x_2)}+\va{f(j_1-1,x_2)})
\\
& \le & 2 \sqrt{\sum_{j_1} \va{f(j_1,x_2) - f(j_1-1,x_2)}^2} \sqrt{ \sum_{j_1} f^2(j_1,x_2)}
\\
& = & 2 \nor{f( \,\cdot, x_2)}_2 \sqrt{\sum_{j_1} \va{\nabla_{x_1}f(j_1,x_2)}^2} := 2 \tilde{f}_1(x_2)\, .
\end{eqnarray*}
Similarly, we have
\begin{equation*}
f^2(x_1,x_2) 
\leq 2 \nor{f(x_1,\cdot\,)}_2 \sqrt{\sum_{j_2} \va{\nabla_{x_2}f(x_1,j_2)}^2} 
:= 2 \tilde{f}_2(x_1) \, .
\end{equation*}
Thus
\[
\sum_{x_1,x_2} f^4(x_1,x_2) \le 4 \pare{\sum_{x_2} \tilde{f}_1(x_2)}  \pare{\sum_{x_1} \tilde{f}_2(x_1)}
\, . 
\]
Since 
\begin{eqnarray*}
 \sum_{x_2} \tilde{f}_1(x_2) & = & \sum_{x_2}  \nor{f( \,\cdot, x_2)}_2 \sqrt{\sum_{j_1} \va{\nabla_{x_1}f(j_1,x_2)}^2}
\\
& \le & \sqrt{  \sum_{x_2} \nor{f( \,\cdot, x_2)}_2^2} 
\sqrt{ \sum_{x_2} \sum_{j_1}  \va{\nabla_{x_1}f(j_1,x_2)}^2}
\\
& \le  & \nor{f}_2  \nor{ \nabla f}_2 \, ,
\end{eqnarray*}
and the same being true for $\sum_{x_1} \tilde{f}_2(x_1)$, it follows that 
\[ \nor{f}_4^4 \le 4 \nor{f}_2^2   \nor{ \nabla f}_2^2 \, ,
\]
which proves \refeq{GNineq2} with $C=2$.
%
%
%

\end{document}